\documentclass{amsart}
\usepackage{amsfonts,latexsym,amssymb,amsmath,amsthm,color}
\usepackage{enumerate}
\usepackage{captdef}
\usepackage[dvips]{graphicx} 
\usepackage{enumerate,graphicx,graphpap}
\usepackage[matrix]{xy}
\usepackage{cases}
\input xy
\xyoption{all} \xyoption{poly}

\theoremstyle{plain}
\newtheorem{teor}{Theorem}[section]
\newtheorem{lema}[teor]{Lemma}

\newtheorem{prop}[teor]{Proposition}

\theoremstyle{definition}
\newtheorem{defi}{Definition}[section]
\newtheorem{eje}{Example}[section]
\newtheorem{nota}[teor]{Remark}

\newtheoremstyle{teoremacita}
{3pt}
{3pt}
{\itshape}
{}
{\bfseries}
{}
{ }
{\thmname{#1}\thmnumber{ #2'}\thmnote{ #3}.}
\theoremstyle{teoremacita} \newtheorem*{teor*}{}

\newcommand{\be}{\begin{enumerate}}
\newcommand{\ee}{\end{enumerate}}

\newcommand{\bi}{\begin{itemize}}
\newcommand{\ei}{\end{itemize}}

\def\N{\mathbb N}
\def\Z{\mathbb Z}

\def\R{\mathbb R}
\def\C{\mathbb C}

\def\e{\varepsilon}

\def\d{\partial}

\begin{document}

\title[Tauberian properties for monomial summability]{Tauberian properties for monomial summability with applications to Pfaffian systems}

\author{Sergio A. Carrillo}
\address[Sergio A. Carrillo]{Dpto. \'{A}lgebra, An\'{a}lisis Matem\'{a}tico, Geometr\'{\i}a y Topolog\'{\i}a \\
Facultad de Ciencias, Universidad de Valladolid \\
Campus Miguel Delibes\\
Paseo de Bel\'{e}n, 7\\
47011 Valladolid - Spain}
\email{sergio.carrillo@agt.uva.es}

\author{Jorge Mozo-Fern\'{a}ndez}
\address[Jorge Mozo Fern\'{a}ndez]{Dpto. \'{A}lgebra, An\'{a}lisis Matem\'{a}tico, Geometr\'{\i}a y Topolog\'{\i}a \\
Facultad de Ciencias, Universidad de Valladolid \\
Campus Miguel Delibes\\
Paseo de Bel\'{e}n, 7\\
47011 Valladolid - Spain}
\email{jmozo@maf.uva.es}

\thanks{First author was supported by an FPI grant (call 2012/2013) conceded by Universidad de Valladolid, Spain.\\ Both authors partially supported by the Ministerio de Econom\'{\i}a y Competitividad from Spain, under the Project ``\'{A}lgebra y Geometr\'{\i}a en Din\'{a}mica Real y Compleja III" (Ref.: MTM2013-46337-C2-1-P)}


\date{\today}

\begin{abstract}
In this paper we will show that monomial summability processes with respect to different monomials are not compatible, except in the (trivial) case of a convergent series. We will apply this fact to the study of solutions of Pfaffian systems with normal crossings, focusing in the implications of the complete integrability condition on these systems.
\end{abstract}

\maketitle

\section{Introduction}

Monomial summability was introduced in \cite{Monomial summ} and used in order to study summability properties of solutions of a class of singularly perturbed differential systems (doubly singular differential systems), that are in some sense a generalization of the singularly perturbed differential equations studied by Canalis-Durand, Ramis, Sch\"{a}fke and Sibuya in \cite{CDRSS}. It is a notion that needs further development in order to treat other problems. For instance, monomial summability depends on a particular monomial $x_1^px_2^q$, and on the order of summability. Recall that in the classical, one variable case, it is well known that if a formal power series is both $k_1-$ and $k_2-$ summable, then it is convergent. In fact, J.-P. Ramis, in the seminal paper \cite{Ramis1} says that \emph{[...] la n\'{e}cessit\'{e} de ``m\'{e}langer'' divers processus de $k$-sommation (avec des $k$ diff\'{e}rents) est \`{a} priori \'{e}vidente pour des raisons alg\'{e}briques \'{e}l\'{e}mentaires [$\cdots$]}. The result is shown, for instance in \cite{Mar-R}, and it is the starting point of the fruitful theory of multisummability in one variable, and the key for establishing many results concerning solutions of systems of differential equations with irregular singularities. So, it turns out that a study of the relationship between monomial summability with respect to different monomials will be useful to treat new problems.

We will explore this situation in the present paper. One of the main results will be the following one:

\begin{teor*}[Theorem \ref{series non monomial sumable}]
Let $x_1^{p_j}x_2^{q_j}$ be monomials, let $k_j>0$ and let $\hat{f}_j\in E\{x_1,x_2\}^{(p_j,q_j)}_{1/k_j}$ be divergent series, for $j=1,...,n$. Then  $\hat{f}_0=\hat{f}_1+\cdots+\hat{f}_n$ is $k_0-$summable in $x_1^{p_0}x_2^{q_0}$ if and only if $p_j/p_0=q_j/q_0=k_0/k_j$ for all $j=1,...,n$.
\end{teor*}

Note that some conditions on the monomials have to be imposed in order to avoid trivial situations. The technique of proof will be induction and blowing up the series, arriving at a situation easier to work with. In some sense, this result is a first step to motivate the introduction of a theory of monomial multisummability, an issue that will be treated in a future work of the authors. 

This result is going to be applied in this paper to analyze Pfaffian systems with normal crossings, i.e., systems as
$$
\begin{cases}
x_2^q\hspace{0.1cm}x_1^{p+1}\hspace{0.075cm}\frac{\d \mathbf{y}}{\d x_1} & =f_1(x_1,x_2, \mathbf{y}),\\
x_1^{p'}x_2^{q'+1}\frac{\d \mathbf{y}}{\d x_2} & =f_2(x_1,x_2, \mathbf{y}),
\end{cases}
$$

These systems were studied by H. Majima \cite{Majima} under complete integrability condition, from the point of view of strong asymptotic expansions. G\'{e}rard and Sibuya \cite{Gerard Sibuya} studied a similar system, when $q=p'=0$. Their main result is that, under complete integrability hypothesis, and assuming the invertibility of the linear parts of the system, the only formal solution is in fact convergent (Theorem 4.1 below). This result can be stated as a theorem on summability in the sense that a series $\hat{f} (x_1,x_2)$ $k_1-$summable in $x_1$ and $k_2-$summable in $x_2$, is in fact convergent.

Complete integrability condition is a very restrictive hypothesis in our situation. In fact, we will show that for a great number of values of $(p,q,p',q')$, complete integrability condition forbids the linear parts of the system from being invertible. Moreover, we conjecture that this is always the case, i.e., you can't assume that any of the linear parts are invertible if $(p,q)\neq (p',q')$. Apparently, this fact hadn't been noticed before by the different authors that studied these systems. So, we are forced to explore summability properties assuming that a formal solution exists, but without assuming complete integrability condition.

The plan of the paper is as follows: in Section 2, the main notions of summability, both in one complex variable, and in the monomial case, are recalled. Tauberian Theorems about the relationship among summability with respect to different monomials, are shown in Section 3. Finally, Section 4 will be devoted to the application of the previous results to the study of Pfaffian systems, and to a precise analysis of the implications of assuming complete integrability condition in these systems.

\section{Summability}

We recall briefly the basic notions of asymptotic expansions, Gevrey asymptotic expansions and summability in the case of one variable and their extensions to the monomial case in two variables, as introduced in the paper \cite{Monomial summ}.

Let $(E,\|\cdot\|)$ be a complex Banach space and $\hat{f}=\sum a_n x^n\in E[[x]]$. We will denote by $\mathcal{O}(U, E)$ (resp. $\mathcal{O}_b(U, E)$) the space of holomorphic $E-$valued maps (resp. holomorphic and bounded $E-$valued maps) defined on an open set $U\subset \C^l$. If $E=\C$ we will simply write $\mathcal{O}(U)$. We also denote by $\N$ the set of natural numbers including $0$, by $\N_{>0}=\N\setminus\{0\}$ and by $D_r\subset\C$ the disk centered at the origin with radius $r$. Consider $f\in\mathcal{O}(V, E)$, where $V=V(a,b,r)=\{x\in\C | 0<|x|<r, a<\text{arg}(x)<b\}$ is an open sector in $\C$. The map $f$ is said to have $\hat{f}$ as \textit{asymptotic expansion }at the origin on $V$ (denoted by $f\sim \hat{f}$ on $V$) if for each of its proper subsectors $V'=V(a',b',r')$ ($a<a'<b'<b$, $0<r'<r$) and each $N\in\N$, there exists $C_N(V')>0$ such that

\begin{equation}\label{def asym classic}
\left\|f(x)-\sum_{n=0}^{N-1} a_n x^n\right\|\leq C_N(V')|x|^N, \hspace{0.3cm} \text{ on } V'.
\end{equation}

If we can take $C_N(V')=C(V')A(V')^N N!^s$, $C(V'), A(V')$ independent of $N$, the asymptotic expansion is said to be of \textit{$s-$Gevrey type} (denoted by $f\sim_s \hat{f}$ on $V$). In this case $\hat{f}\in E[[x]]_s$, where $E[[x]]_s$ denotes the space of $s-$Gevrey series, i.e. there exist $C,A>0$ such that $\|a_n\|\leq CA^nn!^s$, for all $n\in\N$.

\begin{nota}\label{f in x iff f in x^p} For further use, we note that $f\sim \hat{f}$ (or $f\sim_s \hat{f}$) on $V$ is equivalent to have inequalities (\ref{def asym classic}) only for $N=ML$, where $M\in\N^*$ is fixed and $L\in\N$ is arbitrary.
\end{nota}

\begin{nota}\label{equiv f-f_n}We also note that $f\sim\hat{f}$ on $V$ if and only if there is $R>0$ and a sequence $(f_N)_{N\in\N}\subset\mathcal{O}_b(D_R, E) $ satisfying that for every subsector $W$ of $V$ and $N\in\N$ there is $C_N(W)>0$ such that $$\|f(x)-f_N(x)\|\leq C_N(W)|x|^N, \hspace{0.3cm}\text{ on }W\cap D_R.$$ In the same manner, $f\sim_s\hat{f}$ on $V$ if and only if the previous condition is satisfied with $C_N(W)=C(W)A(W)^N N!^s$, for some $C(W), A(W)$ independent of $N$, and there are constants $B,D$ with $\sup_{|x|\leq R}\|f_N(x)\|\leq DB^N N!^s$, for all $N\in\N$. In any case, $\hat{f}$ is given by the limit of the Taylor series at the origin of $f_N$, in the $\mathfrak{m}-$topology of $E[[x]]$.
\end{nota}

Asymptotic expansions are unique, and respect algebraic operations and differentiation. Besides $f\sim_s 0$ on $V$ if and only if for each proper subsector $V'$ of $V$ we can find positive constants $C(V'), A(V')$ such that $$\|f(x)\|\leq C(V')\exp\left(-A(V')/|x|^{1/s}\right), \hspace{0.3cm} \text{ on } V'.$$

From this inequality can  be deduced  Watson's lemma that states that if $f\sim_s 0$ on $V(a,b,r)$ and $b-a> s\pi$ then $f\equiv 0$. This is the key point to define $k-$summability, as is stated below.

\begin{defi}Let  $\hat{f}\in E[[x]]$ be a formal power series, let $k>0$ and let $d$ be a direction.
\begin{enumerate}
\item The formal series $\hat{f}$ is called \textit{$k-$summable on $V=V(a,b,r)$} if $b-a>\pi/k$ and there exists a map $f\in\mathcal{O}(V,E)$ such that $f\sim_{1/k} \hat{f}$ on $V$. If $d$ is the bisecting direction of $V$ then $\hat{f}$ is called \textit{$k-$summable in the direction $d$} and $f$ is called \textit{the $k-$sum} of $\hat{f}$ in direction $d$ (or on $V$).

\item The formal series $\hat{f}$ is called \textit{$k-$summable}, if it is $k-$summable in every direction with finitely many exceptions mod. $2\pi$ (the singular directions).
\end{enumerate}

The algebra of $k-$summable series in the direction $d$ will be denoted by $E\{x\}_{1/k,d}$ and the algebra of $k-$summable series will be denoted by $E\{x\}_{1/k}$.
\end{defi}

The previous notions can be adapted to series in a monomial in several variables (two in our case). The main idea is to let the monomial play the role of principal variable. Let $p,q\in\N_{>0}$ be fixed and consider the monomial $x_1^px_2^q$. For formal power series we may use the filtration of $E[[x_1,x_2]]$ given by the sequence of ideals $(x_1^px_2^q)^N$, and write any $\hat{f}=\sum a_{n,m} x_1^nx_2^m\in E[[x_1,x_2]]$ uniquely as \begin{equation}\label{para definir Tpq}
\hat{f}=\sum_{n=0}^\infty f_n(x_1,x_2)(x_1^px_2^q)^n,
\end{equation}

\noindent where the series $f_n$ are given explicitly by:

\begin{equation}\label{f_n x^pe^q}
f_n(x_1,x_2)=\sum_{m<p\text{ or }j<q}a_{np+m,nq+j}x_1^mx_2^j.
\end{equation}

To ensure that each $f_n$ gives rise to an holomorphic map defined in a common polydisc at the origin it is necessary and sufficient that $\hat{f}\in \mathcal{C}=\bigcup_{r>0} \mathcal{C}_r$, where $\mathcal{C}_r=\mathcal{O}_b(D_r,E)[[x_1]]\cap \mathcal{O}_b(D_r,E)[[x_2]]$. If this is the case then $f_n\in \mathcal{E}^{(p,q)}=\bigcup_{r>0}\mathcal{E}^{(p,q)}_r$, where $\mathcal{E}^{(p,q)}_r$ is the space of holomorphic maps $g\in\mathcal{O}_b(D_r^2,E)$ with $\frac{\d^{n+m}g}{\d x_1^n\d x_2^m}(0,0)=0$ for $n\geq p$ and $m\geq q$. Each $\mathcal{E}^{(p,q)}_r$ becomes a Banach space with the supremum norm $\|g\|_r=\sup_{|x_1|,|x_2|\leq r} \|g(x_1,x_2)\|$.

Define a map $\hat{T}_{p,q}:\mathcal{C}\rightarrow \mathcal{E}^{(p,q)}[[t]]$ defined by $\hat{T}_{p,q}(\hat{f})=\sum_{n=0}^\infty f_n t^n$, using the decomposition (\ref{para definir Tpq}). We recall that $\hat{f}$ is said to be \textit{$s-$Gevrey in the monomial $x_1^px_2^q$} if for some $r>0$, $\hat{T}_{p,q}(\hat{f})\in\mathcal{E}_r^{(p,q)}[[t]]$ and it is a $s-$Gevrey series in $t$. The set of $s-$Gevrey series in the monomial $x_1^px_2^q$ will be denoted by $E[[x_1,x_2]]_s^{(p,q)}$.

We can characterize directly from the growth of the coefficients of a series, when it is of some Gevrey type in a monomial. Indeed, it follows from Cauchy's estimates that $\sum a_{n,m}x_1^nx_2^m\in E[[x_1,x_2]]_s^{(p,q)}$ if and only if there exist constants $C,A$ such that $\|a_{n,m}\|\leq CA^{n+m}\min\{n!^{s/p}, m!^{s/q}\}$, for all $n,m\in\N$. In particular we observe that $E[[x_1,x_2]]_{Ms}^{(Mp,Mq)}=E[[x_1,x_2]]_s^{(p,q)}$ for any $M\in\N_{>0}$ and $\hat{f}\in E\{x_1,x_2\}$ if and only if $\hat{T}_{p,q}(\hat{f})\in \mathcal{E}^{(p,q)}_r\{t\}$, for some $r>0$ by taking $s=0$. We can also relate the Gevrey type in a monomial in terms of another monomial as the next statement shows.

\begin{lema}\label{relating to monomials (formal)}If $\hat{f}\in E[[x_1,x_2]]_s^{(p',q')}$ then $\hat{T}_{p,q}(\hat{f})$ is a $\max\{p/p',q/q'\}s-$Gevrey series in some $\mathcal{E}^{(p,q)}_r$.
\end{lema}

\begin{proof}If $\hat{f}=\sum a_{n,m}x_1^nx_2^m$ and $\|a_{n,m}\|\leq CA^{n+m}\min\{n!^{1/p'}, m!^{1/q'}\}^s$ for all $n,m\in\N$, we can directly estimate the growth of the $f_n$ by means of formula (\ref{f_n x^pe^q}): if $|x_1|, |x_2|<r$ and $rA<1$ we get

\begin{align*}
|f_n(x_1,x_2)|\leq \frac{CA^{n(p+q)}}{1-rA}\left[\sum_{j=0}^{q-1} (nq+j)!^{s/q'}(rA)^j+\sum_{m=0}^{p-1}(np+m)!^{s/p'}(rA)^{q+m}\right].
\end{align*}

\noindent The result follows from
\begin{equation}\label{lim type stirling}\lim_{n\rightarrow\infty} (nl)!^{1/l}/l^n n!=0\end{equation} valid for any integer $l\geq2$.
\end{proof}

In the analytic setting we use \textit{sectors in the monomial $x_1^px_2^q$}, i.e. sets of the form $$\Pi_{p,q}=\Pi_{p,q}(a,b,r)=\left\{(x_1,x_2)\in \C^2 \hspace{0.1cm}|\hspace{0.1cm} 0<|x_1|^p, |x_2|^q<r,\hspace{0.1cm} a<\text{arg}(x_1^px_2^q)<b\right\}.$$

\noindent Here any convenient branch of arg may be used. The number $r$ denotes the \textit{radius}, $b-a$ the \textit{opening} and $(b+a)/2$ the \textit{bisecting direction} of the sector. 

We can construct an operator related to $\hat{T}_{p,q}$ for holomorphic maps defined on monomial sectors. To recall that construction we will focus first in the case $p=q=1$. We shall consider the maps $\pi_1, \pi_2:\C^2\rightarrow\C^2$ given by $\pi_1(x_1,x_2)=(x_1x_2,x_2)$, $\pi_2(x_1,x_2)=(x_1,x_1x_2)$, respectively, corresponding to the point blow-up of the origin in $\C^2$. Note it is sufficient to work with $\pi_1$. For a sector $\Pi_{1,1}(a,b,r)$ is clear that $$\pi_1(\Pi_{1,1}(a,b,r))=\left\{(t,x_2)\in D_{r^2}\times\C \hspace{0.1cm}|\hspace{0.1cm} |t|/r<|x_2|<r \text{, } a<\text{arg}(t)<b \right\}.$$

Any $f\in\mathcal{O}(\Pi_{1,1},E)$, $\Pi_{1,1}=\Pi_{1,1}(a,b,r)$, induces a holomorphic map on $\pi_1(\Pi_{1,1})$ given by $(t,x_2)\mapsto f(t/x_2, x_2)$. For fixed $t$ with $0<|t|<r^2$ the map $x_2\mapsto f(t/x_2,x_2)$ is holomorphic in the annulus $|t|/r<|x_2|<r$ and thus it has a convergent Laurent series expansion on $x_2$:

\begin{equation}\label{f(t/e,e)}
f\left(\frac{t}{x_2},x_2\right)=\sum_{n\in\Z} f_n(t) x_2^n,
\end{equation}

\noindent where the maps $f_n$ are holomorphic on $V=V(a,b,r^2)$. If $f$ is bounded then it induces the holomorphic map $T_{1,1}(f)_\rho:V\rightarrow \mathcal{E}_{\rho}^{(1,1)}$ for all $\rho<r$, by means of the decomposition (\ref{f(t/e,e)}):

\begin{equation}\label{Tf}
T_{1,1}(f)_\rho(t)(x_1,x_2)=\sum_{m=0}^\infty \frac{f_{-m}(t)}{t^m}x_1^m+\sum_{m=1}^\infty f_m(t)x_2^m.
\end{equation}

\noindent More generally, if $f$ satisfies $|f(x_1,x_2)|\leq K(|x_1x_2|)$ on $\Pi_{p,q}$, for some function $K:(0,r^2)\rightarrow\R$ then using Cauchy's estimates we see that:

\begin{equation}\label{inq T11}
\left\|T_{1,1}(f)_\rho(t)\right\|_\rho\leq \frac{K(|t|)}{\left(1-\frac{\rho}{r}\right)^2}, \hspace{0.3cm} \text{ on }V(a,b,\rho).
\end{equation}

\noindent Note that $f$ is determined by $T_{1,1}(f)_\rho$ since $T_{1,1}(f)_\rho(x_1x_2)(x_1,x_2)=f(x_1,x_2)$.

For the general case note that any $f\in\mathcal{O}(\Pi_{p,q},E)$, $\Pi_{p,q}=\Pi_{p,q}(a,b,r)$, can be decomposed uniquely as

\begin{equation}\label{decom f in fij}
f(x_1,x_2)=\sum_{i<p\text{ and }j<q} x_1^ix_2^j f_{ij}(x_1^p,x_2^q),
\end{equation}

\noindent where $f_{ij}\in\mathcal{O}(\Pi_{1,1},E)$ for every $i<p$, $j<q$. For $f$ bounded we can define $T_{p,q}(f)_\rho:V(a,b,\rho^2)\rightarrow \mathcal{E}^{(p,q)}_{\rho}$, $0<\rho<r$, as follows

\begin{align}\label{T_{pq}}
T_{p,q}(f)_\rho(t)(x_1,x_2)&=\sum_{i<p\text{ and }j<q} x_1^ix_2^j T_{1,1}(f_{ij})_\rho(t)(x_1^p, x_2^q).
\end{align}

To see that $T_{p,q}(f)$ is well-defined and thus holomorphic we can even show that if $f$ satisfies $|f(x_1,x_2)|\leq K(|x_1^px_2^q|)$ on $\Pi_{p,q}$, for some function $K:(0,r^2)\rightarrow\R$ then: \begin{equation}\label{inq Tpq}
\left\|T_{p,q}(f)_\rho(t)\right\|_\rho\leq \frac{r^2}{\left(1-(\frac{\rho}{r})^{\frac{1}{p}}\right)\left(1-(\frac{\rho}{r})^{\frac{1}{q}}\right)}\frac{K(|t|)}{|t|}, \hspace{0.3cm} \text{ on }V(a,b,\rho).
\end{equation}

\noindent Indeed, this inequality is deduced from inequality (\ref{inq T11}) and inequalities \begin{equation}\label{inq f_ij}
|\zeta_1\zeta_2f_{ij}(\zeta_1, \zeta_2)|\leq r^{2-i/p-j/q}K(|\zeta_1\zeta_2|), \hspace{0.3cm} \text{ on }\Pi_{1,1},  i<p, j<q.
\end{equation}

\noindent Once again, the map $f$ is determined by $T_{p,q}(f)_\rho$ since $T_{p,q}(f)_\rho(x_1^px_2^q)(x_1,x_2)=f(x_1,x_2)$.

At this point we have the adequate frame to recall the notion of asymptotic expansion in a monomial. Let $f\in\mathcal{O}(\Pi_{p,q},E)$, $\Pi_{p,q}=\Pi_{p,q}(a,b,r)$ and $\hat{f}\in\mathcal{C}$. We will say that \textit{$f$ has $\hat{f}$ as asymptotic expansion at the origin in $x_1^px_2^q$} (denoted by $f\sim^{(p,q)} \hat{f}$ on $\Pi_{p,q}$) if there is $0<r'\leq r$ such that $\hat{T}_{p,q}\hat{f}=\sum f_nt^n\in \mathcal{E}^{(p,q)}_{r'}[[t]]$ and for every proper subsector  $\Pi_{p,q}'=\Pi_{p,q}(a',b',\rho)$ with $0<\rho<r'$ and $N\in\N$ there exists $C_N(\Pi_{p,q}')>0$ such that:

\begin{equation}\label{formula def asym x^pe^q}
\left|f(x_1,x_2)-\sum_{n=0}^{N-1}f_n(x_1,x_2)(x_1^px_2^q)^n\right|\leq C_N(\Pi_{p,q}')|x_1^px_2^q|^N,\hspace{0.3cm} \text{ on } \Pi_{p,q}'.
\end{equation}

The asymptotic expansion is said to be of \textit{$s-$Gevrey type} (denoted by $f\sim^{(p,q)}_s \hat{f}$ on $\Pi_{p,q}$) if it is possible to choose $C_N(\Pi_{p,q}')=C(\Pi_{p,q}')A(\Pi_{p,q}')^N N!^s$ for some constants $C(\Pi_{p,q}')$, $A(\Pi_{p,q}')$ independent of $N$.

There are different characterizations of asymptotic expansions in a monomial, useful in certain situations. For instance, $f\sim^{(p,q)} \hat{f}$ on $\Pi_{p,q}(a,b,r)$ if and only if $T_{p,q}(f)_\rho\sim \hat{T}_{p,q}(\hat{f})$ on $V(a,b,\rho^2)$, for every $0<\rho<r'$. Here $\hat{T}_{p,q}\hat{f}\in\mathcal{E}_{r'}^{(p,q)}[[t]]$ and $0<r'<r$. The result is also valid for asymptotic expansions of $s-$Gevrey type. Indeed, if $f\sim^{(p,q)} \hat{f}$ on $\Pi_{p,q}(a,b,r)$ then applying (\ref{inq Tpq}) to (\ref{formula def asym x^pe^q}) with $K(u)=u^{N+1}$ we obtain

$$\left\|T_{p,q}(f)_{\rho'}(t)-\sum_{n=0}^{N}f_n t^n\right\|_{\rho'}\leq \frac{\rho^2C_{N+1}(\Pi_{p,q}')}{\left(1-(\frac{\rho'}{\rho})^{\frac{1}{p}}\right)\left(1-(\frac{\rho'}{\rho})^{\frac{1}{q}}\right)}|t|^{N},$$

\noindent on $V(a',b',\rho')$ for any $0<\rho'<\rho<r'$ and thus

$$\left\|T_{p,q}(f)_{\rho'}(t)-\sum_{n=0}^{N-1}f_n t^n\right\|_{\rho'}\leq \|f_N\|_{\rho'}|t|^N+\frac{\rho^2C_{N+1}(\Pi_{p,q}')}{\left(1-(\frac{\rho'}{\rho})^{\frac{1}{p}}\right)\left(1-(\frac{\rho'}{\rho})^{\frac{1}{q}}\right)}|t|^{N}.$$

The converse is trivial, i.e. replacing $t$ by $x_1^px_2^q$. We note that for the $s-$Gevrey case we can conclude from $f\sim^{(p,q)}_s \hat{f}$ that $\hat{f}\in E[[x_1,x_2]]_s^{(p,q)}$. In fact, from the previous inequalities it is easy to see that

$$\|f_N\|_{\rho'}\leq \frac{\rho^2C(\Pi_{p,q}')A(\Pi_{p,q}')^NN!^s}{\left(1-(\frac{\rho'}{\rho})^{\frac{1}{p}}\right)\left(1-(\frac{\rho'}{\rho})^{\frac{1}{q}}\right)} \left(A(\Pi_{p,q}')(N+1)^s+\frac{1}{|t|}\right),$$

\noindent where $C_N(\Pi_{p,q}')=C(\Pi_{p,q}')A(\Pi_{p,q}')^N N!^s$ and then fixing $|t|=\rho'/2$ we obtain a Gevrey bound for $f_N$.

Another characterization of monomial asymptotic expansions is given below, this time approximating by holomorphic functions. For further references we state the result as a proposition.

\begin{prop}\label{equiv f iff f_n for x^pe^q}Let $f\in\mathcal{O}(\Pi_{p,q},E)$, $\Pi_{p,q}=\Pi_{p,q}(a,b,r)$, be an holomorphic map. The following assertions are equivalent:
\begin{enumerate}
\item $f\sim^{(p,q)} \hat{f}$ on $\Pi_{p,q}$,

\item There is $R>0$ and a sequence $(f_N)_{N\in\N}\subset \mathcal{O}_b(D_R^2,E)$ such that for every subsector $\Pi_{p,q}'$ of $\Pi_{p,q}$ and $N\in\N$ there is a positive constant $A_N(\Pi_{p,q}')$ such that \begin{equation}\label{f-f_N}|f(x_1,x_2)-f_N(x_1,x_2)|\leq A_N(\Pi_{p,q}')|x_1^px_2^q|^N,\hspace{0.3cm}\text{ on }\Pi_{p,q}'\cap D_R^2.\end{equation}
\end{enumerate}
Analogously, $f\sim^{(p,q)}_s \hat{f}$ on $\Pi_{p,q}$ if and only if (2) is satisfied with $A_N(\Pi_{p,q}')=C(\Pi_{p,q}')A(\Pi_{p,q}')^NN!^s$ for some $C(\Pi_{p,q}'), A(\Pi_{p,q}')$ independent of $N$ and there are positive constants $B,D$ such that $\|f_N\|_R\leq DB^NN!^s$ for all $N\in\N$. In any case, $\hat{f}$ is given by the limit of the Taylor series at the origin of $f_N$, in the $\mathfrak{m}-$topology of $E[[x_1,x_2]]$.
\end{prop}

\begin{proof}We only prove the case of Gevrey asymptotics. If $f\sim_s^{(p,q)} \hat{f}$ on $\Pi_{p,q}$ and $\hat{T}_{p,q}(\hat{f})(t)=\sum g_n t^n$ then $f_N(x_1,x_2)=\sum_{n<N} g_n(x_1,x_2) (x_1^px_2^q)^n$ satisfies the requirements. Conversely, suppose we have such a family. Note that each $T_{p,q}(f_N)_R$ is holomorphic on $D_R$ and has $\hat{T}_{p,q}(f_N)$ as Taylor series at the origin. Let $g_{N+1}=T_{p,q}(f_N)_R$. Applying inequalities (\ref{inq Tpq}) to (\ref{f-f_N}) with $K(u)=u^N$ it follows that

$$\|T_{p,q}(f)_{\rho'}(t)-g_{N+1}\|_{\rho'}\leq \frac{\rho^2A_{N+1}(\Pi_{p,q}')}{\left(1-(\frac{\rho'}{\rho})^{\frac{1}{p}}\right)\left(1-(\frac{\rho'}{\rho})^{\frac{1}{q}}\right)}|t|^N,$$

\noindent in the corresponding sector where $0<|t|<\rho'<\min\{\rho,R\}$. Using Remark \ref{equiv f-f_n} we obtain that $T_{p,q}(f)_{\rho'}$ has an asymptotic expansion $\hat{T}_{p,q}(\hat{f})$ given by the limit of the series $\hat{T}_{p,q}(f_N)$ in the $\mathfrak{m}-$topology of $\mathcal{E}^{(p,q)}_r[[t]]$ and thus $f\sim^{(p,q)}\hat{f}$ on $\Pi_{p,q}$ as we wanted to show. The Gevrey case follows noting that the $(g_{N+1})$ have Gevrey bounds if the $(f_N)$ have them.

\end{proof}


In this context the analogous result to Watson's lemma reads as follows: if  $f\sim_s^{(p,q)} 0$ on $\Pi_{p,q}(a,b,r)$ and $b-a>s\pi$ then $f\equiv 0$. Finally we can recall the definition of $k-$summability in a monomial.

\begin{defi}Let $\hat{f}\in \mathcal{C}$, let $k>0$ and let $d$ be a direction.
\begin{enumerate}
\item The formal series $\hat{f}$ is called \textit{$k-$summable in $x_1^px_2^q$ on $\Pi_{p,q}=\Pi_{p,q}(a,b,r)$} if $b-a>\pi/k$ and there exists a map $f\in\mathcal{O}(\Pi_{p,q},E)$ such that $f\sim_{1/k}^{(p,q)} \hat{f}$ on $\Pi_{p,q}$. If $d$ is the bisecting direction of $\Pi_{p,q}$ then $\hat{f}$ is called \textit{$k-$summable in $x_1^px_2^q$ in the direction $d$} and $f$ is called \textit{the $k-$sum in $x_1^px_2^q$} of $\hat{f}$ in direction $d$ (or on $\Pi_{p,q}$).

\item The formal series $\hat{f}$ is called \textit{$k-$summable in $x_1^px_2^q$}, if it is $k-$summable in $x_1^px_2^q$ in every direction with finitely many exceptions mod. $2\pi$ (the singular directions).
\end{enumerate}

The algebra of $k-$summable series in $x_1^px_2^q$ in the direction $d$ will be denoted by $E\{x_1,x_2\}^{(p,q)}_{1/k,d}$ and the algebra of $k-$summable series in $x_1^px_2^q$ will be denoted by $E\{x_1,x_2\}^{(p,q)}_{1/k}$.
\end{defi}

We have seen than $\hat{f}$ is $k-$summable in $x_1^px_2^q$ (resp. $k-$summable in direction $d$) if and only if  $\hat{T}_{p,q}(\hat{f})_\rho$ is $k-$summable (resp. $k-$summable in direction $d$) for all $\rho$ small enough. With this characterization it is possible to carry on known theorems of summability in our context.

To finish this section we want to give an example of monomial summability based on an example of Poincar\'{e}, see \cite{Sauzin}.

\begin{eje}Consider the series $\hat{f}=\sum a_{n,m} x_1^n x_2^m$, $a_{0,0}=0$ and $a_{n,m}=(-|n-m|)^{\min\{n,m\}}$ for the other terms. We want to study its $1$-summability in $x_1x_2$. It is immediate to check that $\hat{F}(t)(x_1,x_2)=$ $\hat{T}_{1,1}(\hat{f})(t)(x_1,x_2)=$ $\sum (b_n(x_1)+c_n(x_2))t^n$ where:

$$b_0(x_1)=\sum_{m=1}^\infty x_1^m, b_n(x_2)=(-1)^n\sum_{m=0}^\infty m^nx_2^m, \hspace{0.3cm} c_n(x_2)=(-1)^n\sum_{m=1}^\infty m^nx_2^m.$$

To calculate its $1-$sum we apply the Borel-Laplace method (see e.g. \cite{Balser2}): its formal $1-$Borel transform is given by

\begin{align*}
\hat{\mathcal{B}}_1(t\hat{F}(t))(\xi)(x_1,x_2)&=\sum_{n=0}^\infty (b_n(x_1)+c_n(x_2))\frac{\xi^n}{n!}\\
&=\frac{1}{1-x_1e^{-\xi}}+\frac{1}{1-x_2 e^{-\xi}}-2=\varphi(x_1,x_2,\xi).
\end{align*}

For a fixed $(x_1,x_2)$ with $|x_1|, |x_2|<1$, the radius of convergence of the above series is the minimum between $\text{dist}(u,2\pi i\Z)$ and $\text{dist}(v,2\pi i\Z)$, where $x_1=e^u$, $x_2=e^v$ and $\text{Re}(u), \text{Re}(v)<0$. The domain of definition of its analytic continuation $\varphi$ is the set conformed by all the triples $(x_1,x_2,\xi)$ in $\C^3$ such that $e^\xi\neq x_1$ and $e^\xi\neq x_2$.

If we fix $0<\rho<r<1$ it follows that $|\varphi(x_1,x_2,\xi)|\leq \frac{2}{1-\rho/r}$ on $D_\rho^2\times \{\xi\in \C | \text{Re}(\xi)\geq \log(r)\}$. Then for a fixed $(x_{10},x_{20})$ in $D_\rho^2$ the power series $\hat{F}(t)(x_{10},x_{20})$ with complex coefficients is $1-$summable in the classical sense in any direction $d\neq \text{arg}(u-2\pi ik), \text{arg}(v-2\pi ik)$, $k\in\Z$. Then $\hat{F}$ is $1-$summable in the space $\mathcal{E}^{(1,1)}_\rho$ only in directions $d$ included in $(-\pi/2, \pi/2)$. To calculate its $1-$sum, for instance in direction $d=0$, we calculate its Laplace transform given by

$$\frac{1}{t}\int_0^{+\infty} \left(\frac{1}{e^\xi-x_1}+\frac{1}{e^\xi-x_2}\right)e^{\xi-\xi/t} d\xi-2=\sum_{n=1}^\infty \frac{x_1^n+x_2^n}{1+nt}.$$

Thus the $1-$sum $f$ in $x_1x_2$ of $\hat{f}$ in direction $d=0$ is obtained by replacing $t=x_1x_2$ in the above expression, and has domain of definition $\{(x_1,x_2)\in\C^2 \hspace{0.1cm}|\hspace{0.1cm} |x_1|, |x_2|<1 \text{ and } x_1x_2\neq-1/n, n\in\N_{>0}\}$.
\end{eje}

In the example, finding the $k-$sum in $x_1^px_2^q$ was possible studying the $k-$sum of the corresponding series under the operator $\hat{T}_{p,q}$. In general such computations are not so simple but fortunately the sum can also be computed as the $pk-$sum (resp. $qk-$sum) of $\hat{f}$ as a formal power series in $x_1$ (resp. $x_2$) with coefficients holomorphic maps on a sector in $x_2$ (resp. $x_1$) with small enough opening.

\section{Tauberian properties for Summability}

The goal of this section is to describe some tauberian theorems for monomial summability, for instance, relate different levels of summability for different monomials. 

In the context of one variable we have the following two statements that provide tauberian properties for $k-$summability and that we will recover in the monomial case.

\begin{prop}\label{k-sum en todas direcciones es convergent}Let $\hat{f}\in E\{x\}_{1/k}$ have no singular directions. Then $\hat{f}$ is convergent.
\end{prop}

\begin{prop}[Ramis]\label{two different levels of sum implies convergence} Let $0<k<l$ be positive numbers. Then $E[[x]]_{1/l}\cap E\{x\}_{1/k}=E\{x\}_{1/l}\cap E\{x\}_{1/k}=E\{x\}.$
\end{prop}

As for a fixed monomial, summability in the monomial is equivalent to summability in the classical sense, we obtain immediately the following statements.

\begin{prop}\label{no sing directions then convergence for monomials}If $\hat{f}\in E\{x_1,x_2\}^{(p,q)}_{1/k}$ has no singular directions then $\hat{f}$ is convergent.
\end{prop}

\begin{prop}\label{p,q,1/k intersection p,q,1/k'}Let $0<k<l$ be positive real numbers. Then for any monomial $x_1^px_2^q$ we have $E\{x_1,x_2\}^{(p,q)}_{1/k}\cap E\{x_1,x_2\}^{(p,q)}_{1/l}=E\{x_1,x_2\}^{(p,q)}_{1/k}\cap E[[x_1,x_2]]^{(p,q)}_{1/l}=E\{x_1,x_2\}$.
\end{prop}

We can extend the last proposition when the monomials involved are not necessarily the same. Thereby we start by relating summability in a monomial with summability in some power of this monomial.

\begin{prop}\label{p,q y Mp,Mq}Let $k>0$, let $p,q, M\in\N_{>0}$ and let $d$ be a direction. Then $E\{x_1,x_2\}_{1/k,d}^{(p,q)}=E\{x_1,x_2\}_{M/k,Md}^{(Mp,Mq)}$ and in particular $E\{x_1,x_2\}_{1/k}^{(p,q)}=E\{x_1,x_2\}_{M/k}^{(Mp,Mq)}$.
\end{prop}

\begin{proof}Let $\hat{f}$ be a formal power series. If $\hat{T}_{p,q}(\hat{f})(t)=\sum_{n=0}^\infty f_nt^n$ then \begin{equation}\label{T_MP,Mq}\hat{T}_{Mp,Mq}(\hat{f})(\tau)=\sum_{n=0}^\infty g_n\tau^n, \hspace{0.2cm} g_n(x_1,x_2)=\sum_{j=0}^{M-1} f_{Mn+j}(x_1,x_2)(x_1^px_2^q)^j.\end{equation}

Suppose that $f\sim^{(p,q)}_{1/k}\hat{f}$ on $\Pi_{p,q}=$$\Pi_{p,q}(a,b,r)=$$\Pi_{Mp,Mq}(Ma, Mb, r^M)=$$\Pi_{Mp,Mq}$, where $d=(a+b)/2$ and $b-a>\pi/k$. Using (\ref{T_MP,Mq}), inequality (\ref{formula def asym x^pe^q}) for $N=ML$, $L\in\N$ and the limit (\ref{lim type stirling}) for $l=M$ we obtain that $f\sim^{(Mp,Mq)}_{M/k} \hat{f}$ on $\Pi_{Mp,Mq}$. Conversely, assume that $f\sim^{(Mp,Mq)}_{M/k} \hat{f}$ on $\Pi_{p,q}$. By definition we see that for every subsector $\Pi_{p,q}'=\Pi_{p,q}(a',b',\rho)$ of $\Pi_{p,q}$ and $N\in \N$ there are positive constants $C=C(\Pi_{p,q}')$, $A=A(\Pi_{p,q}')$ such that

$$\left|f(x_1,x_2)-\sum_{n=0}^{MN-1} f_n(x_1,x_2)(x_1^px_2^q)^n \right|\leq CA^{MN} N!^{M/k}|x_1^px_2^q|^{MN},\hspace{0.2cm}\text{ on }\Pi'_{p,q}.$$

Applying inequality (\ref{inq Tpq}) with $K(u)=u^{MN}$ we see that

$$\left\|T_{p,q}(f)_{\rho'}(t)-\sum_{n=0}^{MN-1} f_nt^n \right\|_{\rho'}\leq \frac{\rho^2CA^{MN} N!^{M/k}}{\left(1-(\frac{\rho'}{\rho})^{\frac{1}{p}}\right)\left(1-(\frac{\rho'}{\rho})^{\frac{1}{q}}\right)}|t|^{MN-1},$$

\noindent on $V(a',b',\rho^2)$, for all $0<\rho'<\rho$. Since $\hat{f}\in E[[x_1,x_2]]^{(p,q)}_{1/k}$, we may find constants $D\geq C$ and $B\geq\max\{1,A\}$ such that $\|f_n\|_\rho\leq DB^nn!^{1/k}$ and then using the previous inequality for $N$ and $N+1$ we obtain

\begin{align*}
\left\|T_{p,q}(f)_{\rho'}(t)-\sum_{n=0}^{MN-1} f_nt^n \right\|_{\rho'}&\leq\\
\frac{\rho^2DB^{NM}(MN+M)!^{1/k}}{\left(1-(\frac{\rho'}{\rho})^{\frac{1}{p}}\right)\left(1-(\frac{\rho'}{\rho})^{\frac{1}{q}}\right)}&\left(\sum_{j=0}^{M-1} B^{j}\rho^j+A^{M}\rho^{M-1}\right)|t|^{MN},\text{ on }V(a',b',\rho^2).
\end{align*}

An application Remark \ref{f in x iff f in x^p} shows that $T_{p,q}(f)_{\rho}\sim_{1/k} \hat{T}_{p,q}(\hat{f})$ on $V(a,b,\rho^2)$, for all $0<\rho<r'$ and so $f\sim_{1/k}^{(p,q)} \hat{f}$ on $\Pi_{p,q}$, as we wanted to prove.
\end{proof}

The last element we need to generalize Proposition \ref{p,q,1/k intersection p,q,1/k'} is the behavior of monomial asymptotic expansions under point blow-ups, i.e. with the composition with the maps $\pi_1$, $\pi_2$. Both the formal and the analytic behavior are described in the next statements.

\begin{lema}\label{f(xe,e) s-Gevrey in x^pe^(p+q)}Let $\hat{f}\in E[[x_1,x_2]]$ be a formal power series. Then the following assertions are true:
\begin{enumerate}
\item $\hat{f}\in E\{x_1,x_2\}$ if and only if $\hat{f}\circ\pi_1\in E\{x_1,x_2\}$ if and only if $\hat{f}\circ\pi_2\in E\{x_1,x_2\}$.
\item $\hat{f}\in E[[x_1,x_2]]^{(p,q)}_s$ if and only if  $\hat{f}\circ\pi_1\in E[[x_1,x_2]]^{(p,p+q)}_s$ and $\hat{f}\circ\pi_2\in E[[x_1,x_2]]^{(p+q,q)}_s$.
\end{enumerate}
\end{lema}

\begin{proof}We only prove the nontrivial implication in the second statement for $\pi_1$.  Let $\hat{f}=\sum a_{n,m} x_1^nx_2^n$ and write $\hat{f}\circ\pi_1=\sum a'_{n,m} x_1^nx_2^m$ where $a'_{n,m}=a_{n,m-n}$ if $m\geq n$ and $0$ otherwise. If $\|a_{n,m}\|\leq CA^{n+m} \min\{n!^{s/p}, m!^{s/q}\}$ then $\|a'_{n,m}\|\leq CA^m n!^{s\lambda/p}(m-n)!^{s(1-\lambda)/q},$ for all $m\geq n$ and $\lambda$ with $0\leq \lambda \leq 1$. The result follows taking $\lambda=p/(p+q)$.
\end{proof}

\begin{prop}\label{monomial sum and blowups}If $\hat{f}\in E\{x_1,x_2\}^{(p,q)}_{1/k,d}$ has $k-$sum $f$ in direction $d$ then $\hat{f}\circ\pi_1\in E\{x_1,x_2\}^{(p,p+q)}_{1/k,d}$, $\hat{f}\circ\pi_2\in E\{x_1,x_2\}^{(p+q,q)}_{1/k,d}$ and have $k-$sums $f\circ\pi_1$, $f\circ\pi_2$ in direction $d$, respectively.
\end{prop}

\begin{proof}The result follows from Proposition \ref{equiv f iff f_n for x^pe^q}. Indeed, if $(f_N)$ is a family of bounded analytic functions that provide the monomial asymptotic expansion to $f$ then $(f_N\circ\pi_j)$ will provide the asymptotic expansion to $f\circ\pi_j$, $j=1,2$.
\end{proof}

At this point we are ready to enounce and prove the main result so far, comparing summable series in different monomials.

\begin{teor}\label{tauberian general case}Let $x_1^px_2^q$ and $x_1^{p'}x_2^{q'}$ be two monomials and $k, l>0$. The following statements hold:

\begin{enumerate}
\item If $\max\{p/p', q/q'\}< l/k$ then $E\{x_1,x_2\}_{1/k}^{(p,q)}\cap  E[[x_1,x_2]]_{1/l}^{(p',q')}=E\{x_1,x_2\}$.

\item $E\{x_1,x_2\}_{1/k}^{(p,q)}\cap E\{x_1,x_2\}_{1/l}^{(p',q')}=E\{x_1,x_2\}$, except in the case $p/p'=q/q'=l/k$ where $E\{x_1,x_2\}_{1/k}^{(p,q)}=E\{x_1,x_2\}_{1/l}^{(p',q')}$.
\end{enumerate}
\end{teor}

\begin{proof}1. If $\hat{f}\in E\{x_1,x_2\}_{1/k}^{(p,q)}\cap E[[x_1,x_2]]_{1/l}^{(p',q')}$ then by Corollary \ref{relating to monomials (formal)}, $\hat{T}_{p,q}(\hat{f})$ is a $\max\{p/p',q/q'\}/l$-Gevrey series but it is also $k-$summable in some $\mathcal{E}_r^{(p,q)}$. Using Theorem \ref{two different levels of sum implies convergence} we conclude that $\hat{T}_{p,q}(\hat{f})$ and therefore $\hat{f}$ are convergent.

2. Take any $\hat{f}\in E\{x_1,x_2\}_{1/k}^{(p,q)}\cap E\{x_1,x_2\}_{1/l}^{(p',q')}$. If $p/p'=q/q'$ the result follows from Propositions \ref{p,q,1/k intersection p,q,1/k'} and \ref{p,q y Mp,Mq}. The cases $\max\{p/p', q/q'\}< l/k$ and $l/k<\min\{p/p', q/q'\}$ follow from (1).

Finally the case $\min\{p/p', q/q'\}\leq l/k\leq \max\{p/p', q/q'\}$ can be reduced to the previous one using point blow-ups. For instance, if $p/p'<l/k\leq q/q'$ take $N\in\N_{>0}$ such that $$\frac{qk-q'l}{p'l-pk}<N.$$ By Proposition \ref{monomial sum and blowups} we see that $\hat{f}\circ\pi_1^N\in E\{x_1,x_2\}_{1/k}^{(p,Np+q)}\cap E\{x_1,x_2\}_{1/l}^{(p',Np'+q')}$ but the new monomials satisfy $\max\{p/p', Np+q/Np'+q'\}<l/k$ and thus $\hat{f}\circ\pi_1^N$ is convergent. By Lemma \ref{f(xe,e) s-Gevrey in x^pe^(p+q)} $\hat{f}$ is convergent as we wanted to prove.
\end{proof}

The previous theorem can be generalized to any number of monomials to provide examples of series non-$k-$summable in any monomial for any positive real number $k$.

\begin{teor}\label{series non monomial sumable}
Let $x_1^{p_j}x_2^{q_j}$ be monomials, let $k_j>0$ and let $\hat{f}_j\in E\{x_1,x_2\}^{(p_j,q_j)}_{1/k_j}$ be divergent series, for $j=1,...,n$. Then  $\hat{f}_0=\hat{f}_1+\cdots+\hat{f}_n$ is $k_0-$summable in $x_1^{p_0}x_2^{q_0}$ if and only if $p_j/p_0=q_j/q_0=k_0/k_j$ for all $j=1,...,n$.
\end{teor}

\begin{proof}By induction on $n$. If $n=1$ the statement is just Theorem \ref{tauberian general case}, (2). Suppose the statement is true for $n-1$ and let us prove it for $n$. If $p_j/p_0=q_j/q_0=k_0/k_j$ hold for all $j=1,...,n$ then $\hat{f}_0$ is $k_0-$summable in $x_1^{p_0}x_2^{q_0}$. Conversely, assume that $\hat{f_0}\in E\{x_1,x_2\}^{(p_0,q_0)}_{1/k_0}$ and $k_0p_0\leq...\leq k_np_n$ (reindexing if necessary). The following situations cover all the possible cases:

\begin{enumerate}
\item[I.] Suppose the strict inequalities $k_0p_0<k_1p_1<\cdots<k_np_n$ and $k_0q_0<k_1q_1<\cdots<k_nq_n$ hold. Then $\hat{T}_{p_0,q_0}(\hat{f}_0)$ is  $\max\{p_0/p_1, q_0/q_1\}/k_1-$Gevrey and $k_0-$summable in some $\mathcal{E}_r^{(p_0,q_0)}$ and thus $\hat{f_0}$ is convergent by Theorem \ref{two different levels of sum implies convergence}. Using the induction hypothesis on the series $\hat{g}_0=\hat{f}_0-\hat{f}_1$ and $\hat{g}_j=\hat{f}_{j+1}$, $j=1,...,n-1$, we obtain a contradiction.

\item[II.] If $k_0p_0=\cdots=k_np_n$, we compare the numbers $k_jq_j$. When all $k_jq_j$ are different from each other we can assume $k_0q_0<\cdots<k_nq_n$, the series $\hat{f}_j\circ\pi_2$ satisfy the hypothesis in (I) and arguing as in that case we see that $\hat{f_0}\circ \pi_2$ and thus $\hat{f_0}$ are convergent. Using the induction hypothesis we obtain a contradiction. Otherwise $k_iq_i=k_lq_l$ for some $i\neq l$, $E\{x_1,x_2\}_{1/k_i}^{(p_i,q_i)}=E\{x_1,x_2\}_{1/k_l}^{(p_l,q_l)}$ and we can use the induction hypothesis to conclude that $k_0q_0=k_jq_j$ for all $j=1,...,n$, as we wanted to show.

\item[III.] Suppose $k_0p_0<k_np_n$ and let $i_0<i_1<\cdots<i_m$ be the places where we have a strict inequality, i.e. $k_{i_l}p_{i_l}<k_{i_l+1}p_{i_l+1}$ and if $i_l+1\leq j \leq i_{l+1}$ then $k_jp_j=k_{i_{l+1}}p_{i_{l+1}}$, for $l=0,...,m$. Take $N\in\N_{>0}$ satisfying $$N>\max_{0\leq l\leq m} \frac{k_{i_l}q_{i_l}-k_{i_l+1}q_{i_l+1}}{k_{i_l+1}p_{i_l+1}-k_{i_l}p_{i_l}},$$ and consider the series $\hat{f_j}\circ\pi_1^N\in E\{x_1,x_2\}_{1/k_j}^{(p_j,q_j')}$, $j=0,...,n$, where $q_j'=Np_j+q_j$. By the election of $N$ we have $k_{i_l}q_{i_l}'<k_{i_l+1}q_{i_l+1}'$ for all $l=0,...,m$ and the order relations between $k_{i_l+1}q_{i_l+1}',..., k_{i_{l+1}}q_{i_{l+1}}'$ are the same as the ones between $k_{i_l+1}q_{i_l+1},..., k_{i_{l+1}}q_{i_{l+1}}$. If for some $l$ two numbers among $k_{i_l+1}q_{i_l+1}'$,$...$, $k_{i_{l+1}}q_{i_{l+1}}'$ are equal, the corresponding spaces coincide and we can use the induction hypothesis to get a contradiction. If not, all the numbers $k_{i_l+1}q_{i_l+1}',..., k_{i_{l+1}}q_{i_{l+1}}'$ are different, for all $l=0,1,...,m$ and we can assume that $k_{i_l+1}q_{i_l+1}'<\cdots<k_{i_{l+1}}q_{i_{l+1}}'$, for all $l=0,1,...,m$, where all the inequalities are strict. In other words, we have arrive to the situation $k_0q_0'<\cdots <k_nq_n'$, where all the inequalities are strict. Therefore the series $\hat{f_j}\circ\pi_1^N\circ\pi_2$ satisfy the hypothesis in (I) and arguing as in that case we see that $\hat{f_0}\circ\pi_1^N\circ\pi_2$ and thus $\hat{f_0}$ are convergent. Using the induction hypothesis we obtain a contradiction.
\end{enumerate}

Since the only non-contradictory case occurs when $p_j/p_0=q_j/q_0=k_0/k_j$ for all $j=1,...,n$ the statement is true for $n$ and the result follows from the principle of induction.
\end{proof}

Given $\hat{f}\in\mathcal{C}$, we can also consider the situation when $\hat{f}$ is summable as a series in $x_1$ or $x_2$, with coefficients in $\mathcal{O}_b(D_R,E)$ for some $R>0$ and compare the summability phenomena we have at our disposal. It turns out again that all the methods are incompatible and the proofs can be reduced to Theorem \ref{tauberian general case} using blow-ups. The key statement is the following.

\begin{prop}\label{summ and blowups}If $\hat{f}\in \mathcal{O}_b(D_R,E)\{x_1\}_{1/k,d}$ has $k-$sum $f$ in direction $d$ then $\hat{f}\circ\pi_1\in E\{x_1,x_2\}^{(1,1)}_{1/k,d}$, $\hat{f}\circ\pi_2\in \mathcal{O}_b(D_R,E)\{x_1\}_{1/k,d}$ and have $k-$sums $f\circ\pi_1$, $f\circ\pi_2$ in direction $d$, respectively.
\end{prop}

\begin{proof}Let $\hat{f}=\sum f_{n*}(x_2)x_1^n$, where $f_{n*}\in \mathcal{O}_b(D_R,E)$. If $f\sim_{1/k}\hat{f}$ on $V(a,b,r)$, this means that for every subsector $V'$ of $V$ and every $N\in\N$ we can find positive $C,A$ such that

$$\left\|f(x_1,x_2)-\sum_{n=0}^{N-1} f_{n*}(x_2)x_1^n\right\|\leq CA^NN!^{1/k}|x_1|^N, \hspace{0.3cm}\text{ on }V'\times D_R.$$

If we consider $\Pi_{1,1}=\Pi_{1,1}(a,b,r')$, $r'=\min\{r,R\}$ and a subsector $\Pi_{1,1}'$ then $$\left\|f\circ\pi_1(x_1,x_2)-\sum_{n=0}^{N-1} f_{n*}(x_2)x_1^nx_2^n\right\|\leq CA^NN!^{1/k}|x_1x_2|^N, \hspace{0.3cm}\text{ on } \Pi_{1,1}',$$

\noindent and thus $f\circ\pi_1\sim^{(1,1)}_{1/k} \hat{f}\circ\pi_1$ on $\Pi_{1,1}$. In particular, if $b-a>\pi/k$ and $d=(b+a)/2$ then $f$ is $k-$summable in $x_1x_2$ in direction $d$. For $\pi_2$, if we consider the sector $V(a,b,r'')$, $r''=r/R$, and a subsector $V'$ then $$\left\|f\circ\pi_2(x_1,x_2)-\sum_{n=0}^{N-1} f_{n*}(x_1x_2)x_1^n\right\|\leq CA^NN!^{1/k}|x_1|^N, \hspace{0.3cm}\text{ on } V'\times D_R.$$ The condition $f\circ\pi_2\sim_{1/k} \hat{f}\circ\pi_2$ on $V(a,b,r'')$ follows from Remark \ref{equiv f-f_n} and the proposition as been proved.
\end{proof}

\begin{teor}\label{tauberian for one variable and monomials}Let $x_1^px_2^q$ be a monomial and let $k, l>0$. The following statements hold:

\begin{enumerate}
\item $\mathcal{O}_b(D_r,E)\{x_1\}_{1/k}\cap \mathcal{O}_b(D_r,E)\{x_2\}_{1/l}=E\{x_1,x_2\}$.

\item $\mathcal{O}_b(D_r,E)\{x_j\}_{1/k}\cap E\{x_1,x_2\}_{1/l}^{(p,q)}=E\{x_1,x_2\}$, for $j=1,2$.
\end{enumerate}
\end{teor}

\begin{proof}For (1) take $\hat{f}\in\mathcal{O}_b(D_r,E)\{x_1\}_{1/k}\cap \mathcal{O}_b(D_r,E)\{x_2\}_{1/l}$. Then by Propositions \ref{monomial sum and blowups}, \ref{summ and blowups} and \ref{tauberian general case}, (2). we conclude that $\hat{f}\circ\pi_1\circ\pi_2\in E\{x_1,x_2\}_{1/k}^{(2,1)}\cap E\{x_1,x_2\}_{1/l}^{(1,1)}=E\{x_1,x_2\}$ and clearly $\hat{f}$ is convergent. The proof for (2) is analogous.
\end{proof}

\begin{nota}The fact that a series in $E[[x_1,x_2]]$, $k-$summable in $x_1$ and $l-$summable in $x_2$ is convergent can be also deduced directly from the definitions, arguing as in \cite{Sibuya1}. As a matter of fact, Y. Sibuya considered only series, solution of certain Pfaffian systems that we will recall in the next section.
\end{nota}

\section{Applications to Pfaffian systems}

The natural setting to apply the tauberian properties we have obtained is in the realm of \textit{Pfaffian systems with normal crossings}, i.e. systems of the form

\begin{subnumcases}\empty\label{EPNL x}
x_2^q\hspace{0.1cm}x_1^{p+1}\hspace{0.075cm}\frac{\d \mathbf{y}}{\d x_1}=f_1(x_1,x_2, \mathbf{y}),\\
\label{EPNL e}x_1^{p'}x_2^{q'+1}\frac{\d \mathbf{y}}{\d x_2}=f_2(x_1,x_2, \mathbf{y}),
\end{subnumcases}

\noindent where $p, q, p', q'\in\N_{>0}$, $\mathbf{y}\in \C^l$, and $f_1, f_2$ are holomorphic functions defined on a neighborhood of the origin in $\C\times\C\times\C^l$. If $f_1(x_1,x_2,\mathbf{0})=f_2(x_1,x_2,\mathbf{0})=\mathbf{0}$ and the functions $f_1, f_2$ satisfy the following \textit{integrability condition} on their domains of definition:

$$-qx_1^{p'}x_2^{q'}f_1(x_1,x_2,\mathbf{y})+x_1^{p'}x_2^{q'+1}\frac{\d f_1}{\d x_2}(x_1,x_2,\mathbf{y})+\frac{\d f_1}{\d \mathbf{y}}(x_1,x_2,\mathbf{y})f_2(x_1,x_2,\mathbf{y})=$$
\begin{equation}\label{int condition nonlinear}
-p'x_1^{p}x_2^{q}f_2(x_1,x_2,\mathbf{y})+x_1^{p+1}x_2^{q}\frac{\d f_2}{\d x_1}(x_1,x_2,\mathbf{y})+\frac{\d f_2}{\d \mathbf{y}}(x_1,x_2,\mathbf{y})f_1(x_1,x_2,\mathbf{y}),
\end{equation}

\noindent then the system will be referred as \textit{completely integrable Pfaffian system with normal crossings}. The normal crossing refers to the singular locus $x_1x_2=0$ where the differential equations change to implicit ones. If we write $A(x_1,x_2)=\frac{\d f_1}{\d \mathbf{y}}(x_1,x_2,\mathbf{0})$ and $B(x_1,x_2)=\frac{\d f_2}{\d \mathbf{y}}(x_1,x_2,\mathbf{0})$ then they satisfy

\begin{equation}\label{EP integrability AB}
x_1^{p'}x_2^{q'}\left(x_2\frac{\d A}{\d x_2}-qA\right)-x_1^{p}x_2^{q}\left(x_1\frac{\d B}{\d x_1}-p'B\right)+[A,B]=\mathbf{0}, \hspace{0.2cm}\text{ on } D_r^2\text{ for some }r>0.
\end{equation}

Here $[\cdot,\cdot]$ denotes the usual Lie bracket of matrices. By taking $x_1=0$ and $x_2=0$ we conclude that  $[A(0,0), B(0,0)]=\mathbf{0}$, $[A(x_1,0), B(x_1,0)]=\mathbf{0}$ and $[A(0,x_2),B(0,x_2)]=\mathbf{0}$ on $D_r^2$.

\begin{nota}We remark that the integrability condition is preserved by point blow-ups. In fact, it is straightforward to check that the pull-back of system (\ref{EPNL x}), (\ref{EPNL e}) under the map $\pi_1^N(t_1,t_2)=(t_1t_2^N,t_2)$, $N\in\N$, (resp. the map $\pi_2^M(t_1,t_2)=(t_1,t_1^Mt_2)$, $M\in\N$) is given by

\begin{subnumcases}\empty\label{EPNL x pi1N}
t_2^{Np+q}\hspace{0.1cm}t_1^{p+1}\hspace{0.075cm}\frac{\d \mathbf{y}}{\d t_1}=f_1(t_1t_2^N,t_2, \mathbf{y}),\\
\label{EPNL e pi1N}t_1^{p'}t_2^{Np'+q'+1}\frac{\d \mathbf{y}}{\d t_2}=f_2(t_1t_2^N,t_2, \mathbf{y})+Nt_1^{p'-p}t_2^{N(p'-p)+q'-q}f_1(t_1t_2^N,t_2, \mathbf{y}).
\end{subnumcases}

\noindent resp.

\begin{subnumcases}\empty\label{EPNL x pi2M}
t_2^{q}\hspace{0.1cm}t_1^{p+Mq+1}\hspace{0.075cm}\frac{\d \mathbf{y}}{\d t_1}=f_1(t_1,t_1^Mt_2, \mathbf{y})+Mt_1^{p-p'+M(q-q')}t_2^{q-q'}f_2(t_1t_2^N,t_2, \mathbf{y}),\\
\label{EPNL e pi2M}t_1^{p'+Mq}t_2^{q'+1}\frac{\d \mathbf{y}}{\d t_2}=f_2(t_1,t_1^Mt_2, \mathbf{y}).
\end{subnumcases}
\end{nota}

To motivate the results we are going to present we start by commenting a theorem due to R. G\'{e}rard and Y. Sibuya on solutions of completely integrable Pfaffian systems with irregular singular points.

\begin{teor}[G\'{e}rard-Sibuya]\label{GerardSibuya theorem} Consider the completely integrable Pffafian system (\ref{EPNL x}), (\ref{EPNL e}), with $q=p'=0$. If $\frac{\d f_1}{\d \mathbf{y}}(0,0,\mathbf{0})$ and $\frac{\d f_2}{\d \mathbf{y}}(0,0,\mathbf{0})$ are invertible then the Pfaffian system admits a unique analytic solution $\mathbf{y}$ at the origin such that $\mathbf{y}(0,0)=0$.
\end{teor}

\begin{proof}[Sketch of proof] The first proof of Theorem \ref{GerardSibuya theorem} can be found in  \cite{Gerard Sibuya}. Due to the nature of the result Y. Sibuya reproved it with different methods, see \cite{Sibuya1} for a proof using summability theory and see \cite{Sibuya3}, \cite{Sibuya2} for a proof in the linear case using algebraic tools. For a more recent proof the reader may also consult \cite{Sanz}. Let us sketch a proof here for completeness, Let $\hat{\mathbf{y}}$ be the unique formal power solution of equation (\ref{EPNL x}). Using the integrability condition it follows that $\hat{\mathbf{y}}$ is also solution of equation (\ref{EPNL e}). By the general theory of summability, $\hat{\mathbf{y}}$ is $p-$summable in the variable $x_1$ and $q'-$summable in the variable $x_2$. Then by Theorem \ref{tauberian for one variable and monomials} $\hat{\mathbf{y}}$ is convergent.
\end{proof}

In our context, the main result obtained by the authors in \cite{Monomial summ} is the following statement on summability of formal solutions of doubly singular equations.

\begin{teor}\label{nonlinear sol is 1 sumable in xe}Consider the singularly perturbed differential equation $$\e^qx^{p+1}\frac{d\mathbf{y}}{dx}=f(x,\e,\mathbf{y}),$$ where $\mathbf{y}\in\C^l$, $p,q\in\N_{>0}$, $f$ analytic in a neighborhood of $(0,0,\mathbf{0})$ and $f(0,0,\mathbf{0})=\mathbf{0}$. If $\frac{\d f}{\d \mathbf{y}}(0,0,\mathbf{0})$ is invertible then the previous equation has a unique formal solution $\hat{\mathbf{y}}$. Furthermore it is $1-$summable in $x^p\e^q$.
\end{teor}

We shall expect that this result joint with Theorem \ref{tauberian general case} lead to a direct generalization of Theorem \ref{GerardSibuya theorem}. Unfortunately the hypothesis of complete integrability imposes serious restrictions on $f_1$ and $f_2$, in particular in their linear parts, and the conditions imposed to generalize the theorem are not satisfied. Indeed, a closer look to equation (\ref{int condition nonlinear}) leads us to the following result.

\begin{teor}\label{spectrum CI systems}Consider the Pfaffian system (\ref{EPNL x}), (\ref{EPNL e}). If it is completely integrable then the following assertions hold:
\begin{enumerate}
\item If $p'<p$ or $q'<q$ then $A(0,0)$ 
is nilpotent.

\item If $p<p'$ or $q<q'$ then $B(0,0)$ 
is nilpotent.

\item If $p=p'$ and $q=q'$, for every eigenvalue $\mu$ of $B(0,0)$ there is an eigenvalue $\lambda$ of $A(0,0)$ such that $q\lambda=p\mu$. The number $\lambda$ is an eigenvalue of $A(0,0)$, when restricted to its invariant subspace $\{v\in\C^l | (B(0,0))-\mu I)^kv=0\text{ for some }k\in\N\}$.
\end{enumerate}
\end{teor}

\begin{proof}We consider first the particular case $p=p'$ and then we reduce the rest of them to this one using for instance point blow-ups. Let us write \begin{equation*}
A(x_1,x_2)=\sum_{n,m\geq 0} A_{n,m}x_1^mx_2^m=\sum_{n\geq0} A_{n*}(x_2)x_1^n=\sum_{m\geq0} A_{*m}(x_1)x_2^m,
\end{equation*} and analogous notations for $B$. Using equation (\ref{EP integrability AB}) we see that

\begin{align}
\label{3}x_2^{q'}\left(x_2 A_{n-p'*}'(x_2)-q A_{n-p'*}(x_2)\right)-(n-p-p')x_2^qB_{n-p*}(x_2)+\sum_{i=0}^n [A_{i*}(x_2),B_{n-i*}(x_2)]=\mathbf{0},
\end{align}

\noindent for all $n\in\N$ and $|x_2|<r$. Let $\mu_0$ be an arbitrary eigenvalue of $B_{0,0}$. We can find $\rho>0$ small enough and $P\in \text{GL}(l,\mathcal{O}(D_\rho))$ such that

$$\overline{B}_{0*}(x_2)=P(x_2)B_{0*}(x_2)P(x_2)^{-1}=\left(\begin{array}{cc}
\overline{B}^{11}_{0}(x_2) & \mathbf{0}\\
\mathbf{0} & \overline{B}^{22}_0(x_2)
\end{array}\right),$$

\noindent so that for every $|x_2|<\rho$, the matrices $\overline{B}^{11}_{0}(x_2)$ and $\overline{B}^{22}_0(x_2)$ have no common eigenvalues and the only eigenvalue of $\overline{B}^{11}_{0}(0)$ is $\mu_0$ (see \cite[Thm. 25.1]{Wasow}). This last property joint with the fact that $B_{0*}$ and $A_{0*}$ commute let us conclude that

$$\overline{A}_{0*}(x_2)=P(x_2)A_{0*}(x_2)P(x_2)^{-1}=\left(\begin{array}{cc}
\overline{A}^{11}_{0}(x_2) & \mathbf{0}\\
\mathbf{0} & \overline{A}^{22}_0(x_2)
\end{array}\right),$$

\noindent where $[\overline{A}^{jj}_0(x_2), \overline{B}^{jj}_0(x_2)]=\mathbf{0}$, $j=1,2$. Let us also write

\begin{align*}
\overline{A}_{1*}(x_2)&=P(x_2)A_{1*}(x_2)P(x_2)^{-1}=\left(\begin{array}{cc}
\overline{A}^{11}_{1}(x_2) & \overline{A}^{12}_{1}(x_2)\\
\overline{A}^{21}_{1}(x_2) & \overline{A}^{22}_1(x_2)
\end{array}\right),\\
\overline{B}_{1*}(x_2)&=P(x_2)B_{1*}(x_2)P(x_2)^{-1}=\left(\begin{array}{cc}
\overline{B}^{11}_{1}(x_2) & \overline{B}^{12}_{1}(x_2)\\
\overline{B}^{21}_{1}(x_2) & \overline{B}^{22}_1(x_2)
\end{array}\right),
\end{align*}

\noindent in the same block-decomposition as $\overline{A}_{0*}(x_2)$ and $\overline{B}_{0*}(x_2)$.

We have two possibilities: First $p=p'=1$, where equation (\ref{3}) for $n=1$ takes the form

\begin{equation*}
x_2^{q'}\left(x_2 A_{0*}'(x_2)-q A_{0*}(x_2)\right)+x_2^qB_{0*}(x_2)+[A_{0*}(x_2),B_{1*}(x_2)]+[A_{1*}(x_2), B_{0*}(x_2)]=\mathbf{0}.
\end{equation*}

\noindent Multiplying by $P(x_2)$ to the left and by $P(x_2)^{-1}$ to the right, the equation obtained in the position $(1,1)$ according to the previous block decomposition is

\begin{equation}\label{new eq p=1=p' n=1}
x_2^{q'+1}(PA_{0*}'P^{-1})^{(1,1)}(x_2)-qx_2^{q'}\overline{A}^{11}_{0}(x_2)+x_2^q\overline{B}^{11}_{0}(x_2)+[\overline{A}^{11}_{0}(x_2),\overline{B}^{11}_{1}(x_2)]+[\overline{A}^{11}_{1}(x_2),\overline{B}^{11}_{0}(x_2)]=\mathbf{0},
\end{equation}

\noindent where $(PA_{0*}'P^{-1})^{(1,1)}$ indicates the matrix in position $(1,1)$ of $PA_{0*}'P^{-1}$. Taking the trace in this equation we see that

$$x_2^{q'+1}\text{tr}\left((PA_{0*}'P^{-1})^{(1,1)}(x_2)\right)-qx_2^{q'}\text{tr}\left(\overline{A}^{11}_{0}(x_2)\right)+x_2^q\text{tr}\left(\overline{B}^{11}_{0}(x_2)\right)=0.$$

\noindent If $q<q'$ we conclude that $\mu_0=0$ and $B_{0,0}$ is nilpotent. If instead $q=q'$ we conclude that

$$q \text{ tr}\left(\overline{A}_0^{1,1}(0)\right)=\text{tr}\left(\overline{B}^{1,1}_0(0)\right)=l_1\mu_0,$$

\noindent where $l_1$ is the size of $\overline{B}^{1,1}_0$. We have two cases here:
\begin{enumerate}
\item $\overline{A}_0^{1,1}(0)$ has only one eigenvalue $\lambda_0$ and then $q\lambda_0=\mu_0$.

\item $\overline{A}_0^{1,1}(0)$ has at least two different eigenvalues. Let $\lambda_0$ be one of them. We apply again the previous splitting process to the matrix $\overline{A}_0^{1,1}(x_2)$ to conclude that $q\lambda_0=\mu_0$.
\end{enumerate}

In the second case $p=p'>1$, we can apply rank reduction, i.e. use the ramification $z_1=x_1^p$. If we write
\begin{equation*}
A(x_1,x_2)=\sum_{i=0}^{p-1} x_1^iA_i(x_1^p,x_2),\hspace{0.5cm} B(x_1,x_2)=\sum_{i=0}^{p-1} x_1^iB_i(x_1^p,x_2),
\end{equation*}

\begin{align*}
\widetilde{A}(z_1,x_2)&=\left(\begin{array}{cccc}
A_0 & z_1 A_{p-1} & \cdots & z_1 A_1\\
A_1 & A_0-z_1x_2^qI& \cdots & z_1 A_2\\
\vdots & \vdots & \ddots & \vdots\\
A_{p-1} & A_{p-2} & \cdots & A_0-(p-1)z_1x_2^q I
\end{array}
\right),\\
\widetilde{B}(z_1,x_2)&=\left(\begin{array}{cccc}
B_0 & z_1 B_{p-1} & \cdots & z_1 B_1\\
B_1 & B_0& \cdots & z_1 B_2\\
\vdots & \vdots & \ddots & \vdots\\
B_{p-1} & B_{p-2} & \cdots & B_0
\end{array}
\right).
\end{align*}


\noindent then $\widetilde{A}, \widetilde{B}\in \text{Mat}(pl\times pl, \C\{z_1,x_2\})$ are obtained from the system (\ref{EPNL x}), (\ref{EPNL e}) as follows: if we write $\mathbf{y}(x_1,x_2)=\mathbf{y}_0(x_1^p,x_2)+x_1\mathbf{y}_1(x_1^p,x_2)+\cdots+ x_1^{p-1}\mathbf{y}_{p-1}(x_1^p,x_2)$ and put $\mathbf{Y}=(\mathbf{y}_0,...,\mathbf{y}_{p-1})^t$ then $\widetilde{A}$ and $\widetilde{B}$ correspond to the linear parts of the Pfaffian system satisfied by $\mathbf{Y}$. It turns out that they also satisfy the differential equation

\begin{equation}\label{int comp t}
z_1x_2^{q'}\left(x_2\frac{\d \widetilde{A}}{\d x_2}-q\widetilde{A}\right)-pz_1x_2^q\left(z_1\frac{\d \widetilde{B}}{\d z_1}-\widetilde{B}\right)+[\widetilde{A},\widetilde{B}]=\mathbf{0}.
\end{equation}

So we are in a similar situation as the initial equation (\ref{EP integrability AB}) and we can apply the case $p=p'=1$. If $q<q'$ then $B_{0,0}$ is nilpotent and if and $q=q'$ then for each eigenvalue $\mu$ of $B_{0,0}$ there is an eigenvalue $\lambda$ of $A_{0,0}$ such that $q\lambda=p\mu$, and the case (3) of the theorem has been proved.

The previous reasonings can be reproduced by starting with an arbitrary eigenvalue of $\lambda_0$ and also by changing the symmetric role of $x_1$ and $x_2$ in equation (\ref{int condition nonlinear}). Thus we have proved that:

\begin{enumerate}
\item[a.] If $p=p'$ and $q<q'$, or $q=q'$ and $p<p'$ then $B_{0,0}$ is nilpotent,
\item[b.] If $p=p'$ and $q'<q$, or $q=q'$ and $p'<p$ then $A_{0,0}$ is nilpotent.
\end{enumerate}

To finish the proof we treat the remaining cases as follows:

\begin{itemize}
\item Assume $p<p'$ and $q<q'$. We multiply equation (\ref{EPNL x}) by $x_2^{q'-q}$ in both sides to get a completely integrable system where we can apply (a) to conclude that $B_{0,0}$ is nilpotent. The same argument works in the case $p'<p$ and $q'<q$ to conclude that $A_{0,0}$ is nilpotent.
\item Assume $p<p'$ and $q'<q$. We take $N\in\N$ such that $N>\frac{q-q'}{p'-p}$ and consider the completely integrable system given by (\ref{EPNL x pi1N}), (\ref{EPNL e pi1N}). By the previous item we conclude that $B_{0,0}$ is nilpotent. Now we take $M\in\N$ such that $M>\frac{p'-p}{q-q'}$ and consider the completely integrable system given by (\ref{EPNL x pi2M}), (\ref{EPNL e pi2M}). By the previous item we conclude that $A_{0,0}$ is nilpotent. The case $p'<p$ and $q<q'$ follows in the same way. Note that in this case, both matrices $A_{0,0}$ and $B_{0,0}$ are nilpotent.
\end{itemize}

\end{proof}

\begin{nota}\label{notaMajima} As a consequence of the previous result, if $(p,q)\neq (p',q')$, at least one of the matrices $A_{0,0}$, $B_{0,0}$ is nilpotent, and in many cases, both are. In fact, we conjecture that, in this situation, always both matrices of the linear parts are nilpotent. It is worth to remark that H. Majima in \cite{Majima} using his theory of strongly asymptotic expansions has studied systems (\ref{EPNL x}), (\ref{EPNL e}) and their generalization to more independent variables in the completely integrable case. We point out that Theorem \ref{spectrum CI systems} reduces certain parts of these studies in the case of two variables exclusively to trivial cases or highly degenerate ones. This fact apparently had not been noticed by H. Majima or by other authors.
\end{nota}

Using the tools we have developed here we can provide information on the solutions of systems in the case of non-integrability . Indeed, we can prove easily the convergence of solutions under generic conditions, when they exist. More specifically we have the following theorem.

\begin{teor}\label{gerard-sibuya generalized}Consider the Pffafian system (\ref{EPNL x}), (\ref{EPNL e}). The following assertions hold:
\begin{enumerate}
\item Suppose the system has a formal solution $\hat{\mathbf{y}}$. If $\frac{\d f_1}{\d \mathbf{y}}(0,0,\mathbf{0})$ and $\frac{\d f_2}{\d \mathbf{y}}(0,0,\mathbf{0})$ are invertible and $x_1^px_2^q\neq x_1^{p'}x_2^{q'}$ then $\hat{\mathbf{y}}$ is convergent.

\item If the system is completely integrable and $\frac{\d f_1}{\d \mathbf{y}}(0,0,\mathbf{0})$ is invertible then the system has a unique formal solution $\hat{\mathbf{y}}$. Moreover $\hat{\mathbf{y}}$ is 1-summable in $x_1^px_2^q$.

\item If the system is completely integrable and $\frac{\d f_2}{\d \mathbf{y}}(0,0,\mathbf{0})$ is invertible then the system has a unique formal solution $\hat{\mathbf{y}}$. Moreover $\hat{\mathbf{y}}$ is 1-summable in $x_1^{p'}x_2^{q'}$.
\end{enumerate}
\end{teor}

\begin{proof}To prove (1) note that if we consider equation (\ref{EPNL x}) as a singularly perturbed ordinary differential equation and $\frac{\d f_1}{\d \mathbf{y}}(0,0,\mathbf{0})$ is invertible then by Theorem \ref{nonlinear sol is 1 sumable in xe} it has a unique formal solution $\hat{\mathbf{y}}_1$, $1-$summable in $x_1^px_2^q$. In the same way, if $\frac{\d f_2}{\d \mathbf{y}}(0,0,\mathbf{0})$ is invertible then (\ref{EPNL e}) has a unique formal solution  $\hat{\mathbf{y}}_2$, $1-$summable in $x_1^{p'}x_2^{q'}$. Since we are assuming that $\hat{\mathbf{y}}=\hat{\mathbf{y}}_1=\hat{\mathbf{y}}_2$ we conclude from Theorem \ref{tauberian general case} that $\hat{\mathbf{y}}$ converges.

The proofs of (2) and (3) are analogous so we only prove (2). If $\frac{\d f_1}{\d \mathbf{y}}(0,0,\mathbf{0})$ is invertible then (\ref{EPNL x}) has a unique solution $\hat{\mathbf{y}}$,  1-summable in $x_1^px_2^q$. To see that $\hat{\mathbf{y}}$ is also a solution of (\ref{EPNL e}) we consider $\hat{\mathbf{w}}=x_1^{p'}x_2^{q'+1}\frac{\d \hat{\mathbf{y}}}{\d x_2}-f_2(x_1,x_2, \hat{\mathbf{y}})$. Then using the integrability condition it is straightforward to check that $\hat{\mathbf{w}}$ is a solution of the linear differential equation with formal coefficients:

$$x_1^{p+1}x_2^q\frac{\d \mathbf{w}}{\d x_1}=\left(p'x_1^px_2^qI_l+\frac{\d f_1}{\d \mathbf{y}}(x_1,x_2,\hat{\mathbf{y}})\right)\mathbf{w},$$

Since $\frac{\d f_1}{\d \mathbf{y}}(0,0,\mathbf{0})$ is invertible, the above equation has a unique formal solution, and since $\mathbf{0}$ is a solution then $\hat{\mathbf{w}}=\mathbf{0}$ as we wanted to show.
\end{proof}

\begin{nota} If the conjecture stated in Remark \ref{notaMajima} is true, the hypothesis of items (2) and (3) of Theorem \ref{gerard-sibuya generalized} cannot be satisfied if $(p,q)\neq (p',q')$. \end{nota}

On the other side the following trivial example exhibits a simple situation of a non-completely integrable system where the hypotheses of the previous theorem hold.

\begin{eje}Consider a constant vector $\mathbf{c}\in\C^l$, arbitrary $p,q,p',q'\in\N_{>0}$ and the Pfaffian system

$$\begin{cases}
x_1^{p+1}x_2^q\hspace{0.18cm}\frac{\d \mathbf{y}}{\d x_1}=\mathbf{y}-\mathbf{c},\\
x_2^{q'+1}x_1^{p'}\frac{\d \mathbf{y}}{\d x_2}=\mathbf{y}-\mathbf{c}.
\end{cases}$$

It has a unique formal solution given by $\hat{\mathbf{y}}=\mathbf{c}$ and it is convergent. Also the system is not completely integrable except by the case $p=p'=q=q'$.
\end{eje}

\bibliographystyle{plaindin_esp}

\end{document}